\documentclass[12pt]{amsart}

\usepackage[english]{babel}
\usepackage{mathrsfs,amssymb}
\usepackage{mathtools}
\usepackage[colorlinks, citecolor = blue]{hyperref}

\usepackage[shortlabels]{enumitem}
\setlist[itemize]{leftmargin=25pt}
\setlist[enumerate]{leftmargin=25pt}

\newtheorem{theorem}{Theorem}[section]
\newtheorem{lemma}[theorem]{Lemma}
\newtheorem{prop}[theorem]{Proposition}

\theoremstyle{definition}
\newtheorem{definition}[theorem]{Definition}

\theoremstyle{remark}

\numberwithin{equation}{section}
\usepackage[colorinlistoftodos,prependcaption,textsize=small]{todonotes}

\DeclareMathOperator*{\esssup}{ess\,sup}
\DeclareMathOperator*{\essinf}{ess\,inf}

\let \la=\lambda
\let \e=\varepsilon
\let \d=\delta

\let \a=\alpha
\let \f=\varphi

\let \O=\Omega
\let \si=\sigma

\let \ga=\gamma

\allowdisplaybreaks

\begin{document}
\title[The maximal operator on variable Lebesgue spaces]
{The maximal operator on variable Lebesgue spaces: an ${\mathcal A}_{\infty}$-characterization}

\author[A.K. Lerner]{Andrei K. Lerner}
\address[A.K. Lerner]{Department of Mathematics,
Bar-Ilan University, 5290002 Ramat Gan, Israel}
\email{lernera@math.biu.ac.il}

\thanks{The author was supported by ISF grant no. 1035/21.}

\begin{abstract}
In this paper we obtain a new boundedness criterion for the maximal operator $M$ on variable exponent spaces $L^{p(\cdot)}$.
It is formulated in terms of the variable exponent analogue of the well known weighted $A_{\infty}$ condition.
\end{abstract}


\maketitle

\section{Introduction}
Let $p(\cdot):{\mathbb R}^n\to [1,\infty]$ be a measurable function. Denote
by $L^{p(\cdot)}$ the space of real-valued measurable functions $f$ on ${\mathbb R}^n$ such
that
$$\|f\|_{L^{p(\cdot)}}:=\inf\Big\{\la>0:\int_{{\mathbb R}^n}\Big(|f(x)|/\la\Big)^{p(x)}dx\le 1\Big\}<\infty.$$

Given a cube $Q\subset {\mathbb R}^n$ and a locally integrable function $f$, denote $\langle f\rangle_Q:=\frac{1}{|Q|}\int_Qf$.
Consider the Hardy--Littlewood maximal operator $M$ defined by
$$Mf(x):=\sup_{Q\ni x}\langle |f|\rangle_Q,$$
where the supremum is taken over all cubes $Q\subset {\mathbb R}^n$ containing the point $x$.

Let ${\mathcal P}$ denote the class of all exponents $p(\cdot)$ for which $M$ is bounded on $L^{p(\cdot)}$.
This class has been studied in a number of works; see, for example, the recent papers \cite{ADK26,L26} as  well as the monographs \cite{CF13,DHHR11},
where one can find motivation and a detailed historical account.

Given a family of pairwise disjoint cubes ${\mathcal F}$ and a locally integrable function $f$, define the averaging operator $T_{\mathcal F}$ by
$$T_{\mathcal F}f:=\sum_{Q\in F}\langle f\rangle_Q\chi_Q.$$

We say that an exponent $p(\cdot)$ satisfies the condition ${\mathcal A}$ if $T_{\mathcal F}$ is bounded on $L^{p(\cdot)}$ uniformly in ${\mathcal F}$.
Since $|T_{\mathcal F}f|\le Mf$, we obviously have that ${\mathcal P}\Rightarrow {\mathcal A}$. A much deeper converse implication was established by Diening \cite{D05} who proved the following result.

\begin{theorem}[{\cite{D05}}]\label{dr} Let $p_->1$ and $p_+<\infty$\footnote{We use the standard notation $p_-:=\essinf p$ and $p_+:=\esssup p$.}. Then $p(\cdot)\in {\mathcal P}$ if and only if $p(\cdot)\in {\mathcal A}$.
\end{theorem}

We say that $p(\cdot)$ satisfies the $A_{p(\cdot)}$ condition if there exists a constant $C>0$ such that for every cube $Q$ and any locally integrable $f$,
$$\langle f\rangle_Q\|\chi_Q\|_{L^{p(\cdot)}}\le C\|f\chi_Q\|_{L^{p(\cdot)}}.$$
This condition is weaker than ${\mathcal A}$ because it corresponds to the case when a family ${\mathcal F}$ in the definition of ${\mathcal A}$ consists of only one cube.

If we consider the weighted $L^p(w)$-norm instead of the $L^{p(\cdot)}$-norm, then the $A_{p(\cdot)}$ condition reduces exactly to the usual Muckenhoupt $A_p$ condition\footnote{We do not recall here the main definitions from the theory of $A_p$ weights and refer the reader to \cite[Ch.~7]{G14}.}, which is necessary and sufficient for the boundedness of $M$ on $L^p(w)$. In contrast, the condition ${\mathcal A}$ in Theorem \ref{dr} cannot be replaced by the $A_{p(\cdot)}$ condition; corresponding counterexamples can be found in \cite[Ex. 4.51]{CF13}, \cite[Th. 5.3.4]{DHHR11}, and~\cite{K08}.

It was later observed in \cite{K07,L10} that the $A_{p(\cdot)}$ condition governs the local boundedness of $M$ on $L^{p(\cdot)}$ (that is, boundedness on compact sets). This observation led to attempts to replace the condition ${\mathcal A}$ in Theorem \ref{dr} with a condition of the form $A_{p(\cdot)}\cap B$, where $B$ controls the behavior of $p(\cdot)$ at infinity. In particular, there is a sufficient condition for $p(\cdot)\in {\mathcal P}$ expressed in the form $A_{p(\cdot)}\cap {\mathcal N}$ (see \cite[Th. 4.52]{CF13} and \cite{ADK26} for a recent extension to unbounded exponents), where ${\mathcal N}$ is the integral decay condition introduced by Nekvinda \cite{N04}. A necessary and sufficient condition, with ${\mathcal N}$ replaced by a more complicated condition~${\mathcal U}_{\infty}$, was obtained in \cite{L26}.

Despite the progress described above, both known conditions characterizing the class ${\mathcal P}$, namely ${\mathcal A}$ and $A_{p(\cdot)}\cap {\mathcal U}_{\infty}$, are not easy to verify.
This motivates the search for new, possibly simpler characterizations.

Our main result is a step toward this goal. It can be viewed as a certain simplification of the condition ${\mathcal A}$; at the same time, it completely bypasses the $A_{p(\cdot)}$ condition.

The inspiration for this result comes from a well-known fact (see, for example, \cite[Ex.~7.3.3]{G14}) in the theory of $A_p$ weights: a weight $w$ belongs to $A_p$ if and only if both $w$ and $\si:=w^{-\frac{1}{p-1}}$ belong to $A_{\infty}$. Notice that $L^{p'}(\si)$ is the dual space of $L^p(w)$. For this reason, $\si$ is called the dual weight of $w$.

Let us now define the variable exponent analogue of the $A_{\infty}$ condition.

\begin{definition}\label{ainfty}
We say that $p(\cdot)\in {\mathcal A}_{\infty}$ if there exist $\la\in (0,1)$ and $C>0$ such that for any family ${\mathcal F}$ of pairwise disjoint cubes,
any non-negative sequence $\{t_Q\}_{Q\in {\mathcal F}}$, and any subset $E_Q\subset Q, Q\in {\mathcal F}$ with $|E_Q|\ge \la|Q|$,
$$\|\sum_{Q\in {\mathcal F}}t_Q\chi_Q\|_{L^{p(\cdot)}}\le C\|\sum_{Q\in {\mathcal F}}t_Q\chi_{E_Q}\|_{L^{p(\cdot)}}.$$
\end{definition}

A similar definition was given by Diening \cite{D05}: in \cite{D05} one assumes that for every $\la\in (0,1)$ there exists $C>0$ such that the above inequality holds, whereas here we require only that there exist $\lambda\in(0,1)$ and $C>0$.

The variable exponent analogue of the weighted duality relation $(L^p(w))'=L^{p'}(\si)$ is
$(L^{p(\cdot)})'=L^{p'(\cdot)}$ (see, e.g., \cite[p. 78]{DHHR11}), where $p'(x):=\frac{p(x)}{p(x)-1}$.

Now our main result reads as follows.

\begin{theorem}\label{omr} Assume that $p_->1$ and $p_+<\infty$. Then $p(\cdot)\in {\mathcal P}$ if and only if $p(\cdot), p'(\cdot)\in {\mathcal A}_{\infty}$.
\end{theorem}

The necessity part of this theorem is elementary. Indeed, suppose that ${\mathcal F}$ is a family of pairwise disjoint cubes. Let $E_Q\subset Q, |E_Q|\ge \la|Q|$ for $Q\in {\mathcal F}$.
Define $\f:=\sum_{Q\in {\mathcal F}}t_Q\chi_{E_Q}$. Then
$$\sum_{Q\in {\mathcal F}}t_Q\chi_Q\le\frac{1}{\la}T_{\mathcal F}\f.$$
From this, since ${\mathcal P}\Rightarrow {\mathcal A}$, we obtain that $p(\cdot)\in {\mathcal A}_{\infty}$. Since the operator $T_{\mathcal F}$ is self-adjoint, the condition ${\mathcal A}$ also means that $T_{\mathcal F}$ is bounded on~$L^{p'(\cdot)}$, and hence $p'(\cdot)\in {\mathcal A}_{\infty}$. Thus, the condition $p(\cdot), p'(\cdot)\in {\mathcal A}_{\infty}$ in Theorem \ref{omr} can be viewed as a testing condition of ${\mathcal A}$ applied to elementary functions of the form $\f$ above.

The sufficiency part of Theorem \ref{omr} is more complicated and is based on a recent result from \cite{L27}. In order to formulate it, we introduce
the $\la$-median maximal operator $m_{\la}$ defined for $\la\in (0,1)$ by
$$m_{\la}f(x):=\sup_{Q\ni x}(f\chi_Q)^*(\la|Q|),$$
where the supremum is taken over all cubes $Q\subset {\mathbb R}^n$ containing the point $x$, and $f^*$ denotes the non-increasing rearrangement of $f$.


By Chebyshev's inequality, $(f\chi_Q)^*(\la|Q|)\le \frac{1}{\la}\langle |f|\rangle_Q$, which shows that $m_{\la}f$ is smaller than $Mf$, in general. However, we have the following result.

\begin{theorem}[{\cite{L27}}]\label{abstrt}
Let $X$ be a Banach function space with the following property: for every $r\ge 1$, the Hardy--Littlewood maximal operator $M$ is bounded on $X^{1/r}$ if and only if $M$ is bounded on $(X^{1/r})'$. Then the following are equivalent:
\begin{enumerate}[(i)]
\item $M$ is bounded on $X$ and on $X'$;
\item $m_{\la}$ is bounded on $X$ and on $X'$ for some $\la\in (0,1)$.
\end{enumerate}
\end{theorem}

Here $\|f\|_{X^{1/r}}:=\||f|^r\|_{X}^{1/r},$ and $X'$ denotes the associate space of $X$.
Now observe that it is an immediate corollary of Theorem \ref{dr} that if $p_->1$ and $p_+<\infty$, then $M$ is bounded on $L^{p(\cdot)}$ if and only if $M$ is bounded on $L^{p'(\cdot)}$
(because the operator $T_{\mathcal F}$ is self-adjoint). An alternative proof of this result can be also found in \cite{L17}. Hence, if we take $X=L^{p(\cdot)}$ in Theorem \ref{abstrt} with $p_->1$ and $p_+<\infty$, then
observing that $(L^{p(\cdot)})^{1/r}=L^{rp(\cdot)}$, we see that the assumption of Theorem \ref{abstrt} is satisfied. Therefore, in order to establish that $M$ is bounded on $L^{p(\cdot)}$ it suffices to show that
$m_{\la}$ is bounded on $L^{p(\cdot)}$ and on $L^{p'(\cdot)}$ for some $\la\in (0,1)$.

Taking into account the above discussion, the sufficiency part of Theorem \ref{omr} is a consequence of Theorem \ref{abstrt} and the following result, which will be proved below.

\begin{theorem}\label{boundml}
Assume that $p_-\ge 1$ and $p_+<\infty$. Then $p(\cdot)\in {\mathcal A}_{\infty}$ if and only if there exists $t\in (0,1)$ such that $m_{t}$ is bounded on $L^{p(\cdot)}$.
\end{theorem}

In this statement, we allow for the possibility that $p_-=1$. This reflects the fact that the operator $m_{\la}$, unlike $M$, is bounded on $L^1({\mathbb R})$.
Note that $m_{\la}$ is also trivially bounded on $L^{\infty}$. However, the assumption $p_+<\infty$ is essential in our method of the proof.

In conclusion, we observe that the proof of Theorem \ref{dr} in \cite{D05} is rather involved. We now see that it can also be obtained in an alternative way. Indeed, it follows from a combination of Theorems \ref{abstrt}, \ref{boundml}, together with the fact that  $M$ is bounded on $L^{p(\cdot)}$ if and only if~$M$ is bounded on $L^{p'(\cdot)}$. (As mentioned above, this follows from Theorem~\ref{dr}, but it can also be proved independently, as in \cite{L17}.) We do not claim that this approach is substantially simpler; rather, it provides an alternative viewpoint on Theorem \ref{dr}.

The paper is organized as follows. Some standard facts are collected in Section 2. In Section 3, we prove a result that is essentially known and can be found in \cite{L17}.
Since this is the main technical tool needed for the proof of Theorem \ref{boundml}, we provide a large part of the argument for the reader’s convenience. The proof of Theorem~\ref{boundml} is given in Section 4.

\section{Preliminaries}
In this section, we collect several standard facts that will be used later.

For a measurable function $f$ on a cube $Q$ and for $\la\in (0,1)$, the non-increasing rearrangement is defined by
$$(f\chi_Q)^*(\la|Q|):=\inf\{\a>0: |\{x\in Q:|f(x)|>\a\}|\le \la|Q|\}.$$
It follows from this definition that
$$|\{x\in Q:|f(x)|>f^*(\la|Q|)\}|\le \la|Q|$$
and
$$|\{x\in Q:|f(x)|\ge f^*(\la|Q|)\}|\ge \la|Q|.$$
Also, the equality
\begin{equation}\label{embeq}
\{x\in {\mathbb R}^n: m_{\la}f(x)>\a\}=\{x\in {\mathbb R}^n:M\chi_{\{|f|>\a\}}(x)>\la\}
\end{equation}
is an immediate corollary of the definition.

For a measurable set $\O\subset {\mathbb R}^n$ of positive measure and for a variable exponent $p(\cdot)$, denote $p_-(\O):=\essinf_{\O}p$ and $p_+(\O):=\esssup_{\O}p$.
Denote $\varrho(f):=\int_{\O}|f(x)|^{p(x)}dx$. The following statement can be found in \cite[p. 25]{CF13}.

\begin{prop}\label{modest} Assume that $p_+(\O)<\infty$. If $\|f\chi_{\O}\|_{L^{p(\cdot)}}>1$, then
$$\varrho(f)^{1/p_+(\O)}\le \|f\chi_{\O}\|_{L^{p(\cdot)}}\le \varrho(f)^{1/p_-(\O)}.$$
If $\|f\chi_{\O}\|_{L^{p(\cdot)}}\le 1$, then
$$\varrho(f)^{1/p_-(\O)}\le \|f\chi_{\O}\|_{L^{p(\cdot)}}\le \varrho(f)^{1/p_+(\O)}.$$
\end{prop}

We will use the following result, which is an immediate corollary of the Besikovitch--Morse covering theorem (see \cite[p. 6]{G75}).

\begin{theorem}\label{cov} Let $r\in (0,1)$, and let $A\subset {\mathbb R}^n$ be bounded. For each $x\in A$, choose a cube $Q_x$ such that $x\in rQ_x$. Then one can extract from the covering
$\{Q_x\}_{x\in A}$ a subcovering $\{Q_j\}$ of $A$, i.e. $A\subset \cup_jQ_j$, such that there exists $N=N(n,r)\in {\mathbb N}$ for which the family $\{Q_j\}$ can be partitioned into $N$
subfamilies of pairwise disjoint cubes.
\end{theorem}

\section{On the ${\mathcal A}_{\infty}$ condition}
Denote by ${\mathcal Q}$ the family of all cubes in ${\mathbb R}^n$ with sides parallel to the axes.
In this section we will prove the following result.

\begin{theorem}\label{impl} Assume that $p(\cdot)\in {\mathcal A_{\infty}}$. Then there exist $\eta,\d\in (0,1)$ and a function $b:{\mathcal Q}\to [0,1]$ such that for every cube $Q$
and any measurable subset $E_Q\subset Q$ with $|E_Q|\ge \eta|Q|$, for all $t\in \big(0, 1/\|\chi_Q\|_{L^{p(\cdot)}}\big]$ we have
\begin{equation}\label{sainf}
\int_Qt^{p(x)}dx\le 2\Big(\int_{E_Q}t^{p(x)}dx+t^{\d}b(Q)\chi_{(0,1)}(t)\Big),
\end{equation}
and, moreover, $\sum_{Q\in {\mathcal F}}b(Q)\le 1$ for every family of pairwise disjoint cubes ${\mathcal F}$.
\end{theorem}

This result is implicitly contained in \cite[Lemma 4.1]{L17}. There, the author considers the weighted $L^{p(\cdot)}(w)$ spaces and the result is proved under a stronger assumption than
$p(\cdot)\in {\mathcal A_{\infty}}$. However, with minor changes the same works for $p(\cdot)\in {\mathcal A_{\infty}}$.
Taking $w=1$  in \cite[Lemma 4.1]{L17}, one can easily derive Theorem \ref{impl} from the resulting statement.
Since the unweighted case, which we consider here, makes the argument simpler, and since Theorem \ref{impl} plays a key role in deriving Theorem~\ref{boundml}, we outline the main steps below for the reader’s convenience.

An important ingredient of the proof is the following reverse H\"older property.

\begin{definition}\label{d3}
We say that $p(\cdot)\in RH$ if there exist $r>1$ and $C>0$ such that for every family $\mathcal S$ of pairwise disjoint cubes
and any non-negative sequence $\{t_Q\}_{Q\in {\mathcal S}}$,
$$\sum_{Q\in {\mathcal S}}\int_Qt_Q^{p(x)}dx\le 1\Rightarrow \sum_{Q\in {\mathcal S}}|Q|\left(\frac{1}{|Q|}\int_Qt_Q^{rp(x)}dx\right)^{1/r}\le C.$$
\end{definition}

It was shown by Diening \cite[Th. 5.6]{D05} that ${\mathcal A}_{\infty}\Rightarrow RH$. A somewhat simplified proof can be also found in \cite[Lemma 5.1]{L17}.
Observe that in both papers the ${\mathcal A}_{\infty}$ condition is used with $\la=1/2$ but the proof works for every $\la\in (0,1)$ with an obvious change which we outline below.

For $Q\in {\mathcal S}$ denote $v_Q(x):=t_Q^{p(x)}, \a_Q:=\langle v_Q\rangle_Q$, and let
$$\O_k(Q):=\Big\{x\in Q:M_Q^dv_Q>\Big(\frac{2^n}{1-\la}\Big)^k\a_Q\Big\},$$
where $M_Q^d$ denotes the dyadic maximal operator restricted to $Q$. By the Calder\'on--Zygmund decomposition, one can write $\O_k(Q)=\cup_jP_j^k(Q)$, where
$P_j^k(Q)$ are pairwise disjoint cubes such that, setting $E_j^k(Q):=P_j^k(Q)\setminus \O_{k+1}(Q)$, we have $|E_j^k(Q)|\ge \la|P_j^k(Q)|$.
Thus, by the ${\mathcal A}_{\infty}$ condition applied to the family ${\mathcal F}:=\{P_j^k(Q), Q\in {\mathcal S}\}$ we obtain
$$\big\|\sum_{Q\in {\mathcal S}}\sum_jt_Q\chi_{P_j^k(Q)}\big\|_{L^{p(\cdot)}}\le C\big\|\sum_{Q\in {\mathcal S}}\sum_jt_Q\chi_{E_j^k(Q)}\big\|_{L^{p(\cdot)}},$$
and therefore,
$$\big\|\sum_{Q\in {\mathcal S}}t_Q\chi_{\O_k(Q)}\big\|_{L^{p(\cdot)}}\le C\big\|\sum_{Q\in {\mathcal S}}t_Q\chi_{\O_k(Q)\setminus \O_{k+1}(Q)}\big\|_{L^{p(\cdot)}}.$$
From this point on, the proof is identical to that in \cite[Lemma 5.1]{L17}.

\begin{lemma}\label{equi} Suppose that $p(\cdot)\in RH$. Then there exist $r,k>1$ and a function $b:{\mathcal Q}\to [0,2]$
such that the following holds: if $\int_Qt^{p(x)}dx\le 1$, then
\begin{equation}\label{eqb}
|Q|\left(\frac{1}{|Q|}\int_Qt^{rp(x)}dx\right)^{1/r}\le k\Big(\int_Qt^{p(x)}dx+b(Q)\Big),
\end{equation}
and, moreover, $\sum_{Q\in {\mathcal F}}b(Q)\le 2$ for every family of pairwise disjoint cubes ${\mathcal F}$.
\end{lemma}

This lemma is contained implicitly in Diening's work \cite{D05}. In an explicit form it can be found in \cite[Lemma 5.2]{L17} (one should only take $w=1$ there).
We mention how to define a function $b$. Let $r>1$ and $C>0$ be the constants from the $RH$-property. Set $k:=2^{1+\frac{p_+}{p_-}}C$. Now for every cube $Q$ define
$t_Q$ to be the supremum of all $t>0$ such that $\int_Qt^{p(x)}dx\le 1$ and
$$|Q|\left(\frac{1}{|Q|}\int_Qt^{rp(x)}dx\right)^{1/r}>k\int_Qt^{p(x)}dx.$$
One shows that for $t=t_Q$, the latter inequality becomes an equality, and we set
$$b(Q):=|Q|\left(\frac{1}{|Q|}\int_Qt_Q^{rp(x)}dx\right)^{1/r}.$$
We refer to \cite[Lemma 5.2]{L17} for further details.

\begin{lemma}\label{RH}
Assume that $p(\cdot)\in {\mathcal A}_{\infty}$. There exist $\ga,C>1$ and $\e>0$ such that
if
$$t\in \big[\min\big(1,1/\|\chi_Q\|^{1+\e}_{L^{p(\cdot)}}\big), \max\big(1,1/\|\chi_Q\|^{1+\e}_{L^{p(\cdot)}}\big)\big],$$
then
$$
\left(\frac{1}{|Q|}\int_Qt^{\ga p(x)}dx\right)^{1/\ga}\le C\frac{1}{|Q|}\int_Qt^{p(x)}dx.
$$
\end{lemma}

This is an analogue of \cite[Lemma 5.3]{L17}. Now the argument requires some simple modifications, and hence we give a complete proof below.

\begin{proof}[Proof of Lemma \ref{RH}] Consider first the case when $\|\chi_Q\|_{L^{p(\cdot)}}\ge 1$. Then $t\le 1$.
Let $\la\in (0,1)$ be a constant from the definition of ${\mathcal A}_{\infty}$. Denote
$$p_{\la}:=(p\chi_Q)^*(\la|Q|).$$

It suffices to show that
\begin{equation}\label{suf}
\left(\frac{1}{|Q|}\int_Qt^{\ga p(x)}dx\right)^{1/\ga}\le Ct^{p_{\la}}.
\end{equation}
Indeed, setting $E_{\la}:=\{x\in Q:p(x)\le p_{\la}\}$, we have $|E_{\la}|\ge (1-\la)|Q|$, and so
$$t^{p_{\la}}\le \frac{1}{1-\la}\frac{1}{|Q|}\int_{E_{\la}}t^{p(x)}dx\le \frac{1}{1-\la}\frac{1}{|Q|}\int_Qt^{p(x)}dx.$$

We will use that ${\mathcal A}_{\infty}\Rightarrow RH$. Let $r>1$ be a constant from the $RH$-property. Take $\ga\in (1,r)$ and let $\e:=\frac{r}{\ga}-1$.
We have
$$
\left(\frac{1}{|Q|}\int_Qt^{\ga p(x)}dx\right)^{1/\ga}=\left(\frac{1}{|Q|}\int_Qt^{\ga(p(x)-p_{\la})}dx\right)^{1/\ga}t^{p_{\la}}.
$$
Next, since $t\in [1/\|\chi_Q\|_{L^{p(\cdot)}}^{1+\e},1]$,
\begin{equation}\label{t}
t^{\ga(p(x)-p_{\la})}\le 1+\|\chi_Q\|_{L^{p(\cdot)}}^{r(p_{\la}-p(x))},
\end{equation}
which, along with the previous estimate, implies
\begin{eqnarray*}
&&\left(\frac{1}{|Q|}\int_Qt^{\ga p(x)}dx\right)^{1/\ga}\\
&&\le t^{p_{\la}}\left(1+\|\chi_Q\|_{L^{p(\cdot)}}^{\frac{r}{\ga}p_{\la}}\Big(\frac{1}{|Q|}\int_Q\Big(\frac{1}{\|\chi_Q\|_{L^{p(\cdot)}}}\Big)^{rp(x)}dx\Big)^{1/\ga}\right).
\end{eqnarray*}

Since $\int_Q\Big(\frac{1}{\|\chi_Q\|_{L^{p(\cdot)}}}\Big)^{p(x)}dx\le 1$, the $RH$-property implies
$$\frac{1}{|Q|}\int_Q\Big(\frac{1}{\|\chi_Q\|_{L^{p(\cdot)}}}\Big)^{rp(x)}dx\le \frac{C}{|Q|^r},$$
and hence
$$
\left(\frac{1}{|Q|}\int_Qt^{\ga p(x)}dx\right)^{1/\ga}\le Ct^{p_{\la}}\left(1+\Big(\|\chi_Q\|_{L^{p(\cdot)}}^{p_{\la}}/|Q|\Big)^{r/\ga}\right).
$$

Therefore, in order to prove (\ref{suf}), it remains to show that
\begin{equation}\label{rem}
\|\chi_Q\|_{L^{p(\cdot)}}^{p_{\la}}/|Q|\le C.
\end{equation}
Denote $\O_{\la}:=\{x\in Q:p(x)\ge p_{\la}\}$. Then $|\O_{\la}|\ge \la|Q|$, and hence, by the ${\mathcal A}_{\infty}$ condition,
$$\|\chi_Q\|_{L^{p(\cdot)}}\le C\|\chi_{\O_{\la}}\|_{L^{p(\cdot)}}.$$

Since $\|\chi_Q\|_{L^{p(\cdot)}}\ge 1$, we have $|Q|\ge 1$. If $|Q|\le 1/\la$, then (\ref{rem}) is trivial. Assume that $|Q|>1/\la$. Then $|\O_{\la}|\ge 1$ and so $\|\chi_{\O_{\la}}\|_{L^{p(\cdot)}}\ge 1$.
By Proposition \ref{modest},
$$\|\chi_{\O_{\la}}\|_{L^{p(\cdot)}}\le |\O_{\la}|^{1/p_-(\O_{\la})}\le |\O_{\la}|^{1/p_{\la}}\le |Q|^{1/p_{\la}},$$
which proves (\ref{rem}). This completes the proof when $\|\chi_Q\|_{L^{p(\cdot)}}\ge 1$.

Suppose now that $\|\chi_Q\|_{L^{p(\cdot)}}\le 1$. Then $t\ge 1$. The proof is essentially the same but with $\la$ changed to $1-\la$ and the reversed role of the sets $E_{1-\la}$ and $\O_{1-\la}$.

As before, it suffices to prove (\ref{suf}) but with $p_{1-\la}$ instead of $p_{\la}$. Indeed, since $t\ge 1$,
$$t^{p_{1-\la}}\le \frac{1}{1-\la}\frac{1}{|Q|}\int_{\O_{1-\la}}t^{p(x)}dx\le \frac{1}{1-\la}\frac{1}{|Q|}\int_Qt^{p(x)}dx.$$
Next, since $t\in [1,1/\|\chi_Q\|_{L^{p(\cdot)}}^{1+\e}]$, the full analogue of (\ref{t}) (with $p_{1-\la}$) holds. The rest of the proof until (\ref{rem}) is the same, and instead of (\ref{rem}) we have to show that
\begin{equation}\label{rem1}
\|\chi_Q\|_{L^{p(\cdot)}}^{p_{1-\la}}/|Q|\le C.
\end{equation}
Since $|E_{1-\la}|\ge \la|Q|$,  by the ${\mathcal A}_{\infty}$ condition,
$$\|\chi_Q\|_{L^{p(\cdot)}}\le C\|\chi_{E_{1-\la}}\|_{L^{p(\cdot)}}.$$
Next, since $\|\chi_{E_{1-\la}}\|_{L^{p(\cdot)}}\le 1$, by Proposition \ref{modest},
$$\|\chi_{E_{1-\la}}\|_{L^{p(\cdot)}}\le |E_{1-\la}|^{1/p_+(E_{1-\la})}\le |E_{1-\la}|^{1/p_{1-\la}}\le |Q|^{1/p_{1-\la}},$$
which proves (\ref{rem1}), and this finishes the proof.
\end{proof}

\begin{proof}[Proof of Theorem \ref{impl}] Let us show that there exist $\ga>1$ and $A,\d>0$ such that
\begin{equation}\label{lus}
|Q|\Big(\frac{1}{|Q|}\int_Qt^{\ga p(x)}dx\Big)^{1/\ga}\le A\Big(\int_{Q}t^{p(x)}dx+t^{\d}b(Q)\chi_{(0,1)}(t)\Big)
\end{equation}
for all $t\le 1/\|\chi_Q\|_{L^{p(\cdot)}}$.

If $t\ge 1$, then (\ref{lus}) follows from Lemma \ref{RH} (since in this case $\|\chi_Q\|_{L^{p(\cdot)}}\le 1$ and hence $t\in [1, 1/\|\chi_Q\|_{L^{p(\cdot)}}^{1+\e}]$ for any $\e>0$).
Therefore, assume that $t<1$.

Let $C,\ga,\e$ be the constants from Lemma \ref{RH}. Let $A>C$ which will be chosen later on.

Consider the case when
\begin{equation}\label{sup}
\frac{1}{|Q|}\int_Qt^{p(x)}dx\le \frac{1}{A}\Big(\frac{1}{|Q|}\int_Qt^{\ga p(x)}dx\Big)^{1/\ga}.
\end{equation}
Then the inequality claimed in Lemma \ref{RH} does not hold which means that $t^{\frac{1}{1+\e}}\le \frac{1}{\|\chi_Q\|_{L^{p(\cdot)}}}$.
Hence, by Lemma \ref{equi},
$$
|Q|\left(\frac{1}{|Q|}\int_Qt^{\frac{r}{1+\e}p(x)}dx\right)^{1/r}\le k\Big(\int_Qt^{\frac{1}{1+\e}p(x)}dx+b(Q)\Big),
$$
By H\"older's inequality along with (\ref{sup}),
\begin{eqnarray*}
\int_Qt^{\frac{1}{1+\e}p(x)}dx&\le& |Q|\Big(\frac{1}{|Q|}\int_Qt^{p(x)}dx\Big)^{\frac{1}{1+\e}}\\
&\le& \frac{1}{A^{\frac{1}{1+\e}}}|Q|\Big(\frac{1}{|Q|}\int_Qt^{\ga p(x)}dx\Big)^{1/(1+\e)\ga}.
\end{eqnarray*}
Take now $A:=\max\big((2k)^{1+\e}, 2C\big)$. Then using that $(1+\e)\ga=r$ and combining the previous estimates, we obtain
$$
|Q|\left(\frac{1}{|Q|}\int_Qt^{\frac{r}{1+\e}p(x)}dx\right)^{1/r}\le 2kb(Q)\le A^{\frac{1}{1+\e}}b(Q),
$$
which implies
$$
|Q|\left(\frac{1}{|Q|}\int_Qt^{rp(x)}dx\right)^{1/r}\le A^{\frac{1}{1+\e}}t^{\frac{\e}{1+\e}p_-}b(Q).
$$

On the other hand, if (\ref{sup}) does not hold, then
$$
|Q|\Big(\frac{1}{|Q|}\int_Qt^{\ga p(x)}dx\Big)^{1/\ga}\le A\int_Qt^{p(x)}dx.
$$
Using that $\ga<r$, from both previous estimates and H\"older's inequality we obtain
$$
|Q|\Big(\frac{1}{|Q|}\int_Qt^{\ga p(x)}dx\Big)^{1/\ga}\le A\int_Qt^{p(x)}+A^{\frac{1}{1+\e}}t^{\frac{\e}{1+\e}p_-}b(Q),
$$
which proves (\ref{lus}).
\enlargethispage{5mm}
Take now $\eta\in (0,1)$ such that $A(1-\eta)^{1/\ga'}=1/2$. Suppose that $E\subset Q$ and $|E|\ge \eta|Q|$. By H\"older's inequality,
$$\int_{Q\setminus E}t^{p(x)}dx\le (1-\eta)^{1/\ga'}|Q|\Big(\frac{1}{|Q|}\int_Qt^{\ga p(x)}dx\Big)^{1/\ga}.$$
This, together with (\ref{lus}), implies
$$\int_Qt^{p(x)}dx\le \int_Et^{p(x)}dx+\frac{1}{2}\Big(\int_Qt^{p(x)}dx+t^{\d}b(Q)\chi_{(0,1)}(t)\Big),$$
from which
$$\int_Qt^{p(x)}dx\le 2\Big(\int_Et^{p(x)}dx+t^{\d}\frac{b(Q)}{2}\chi_{(0,1)}(t)\Big),$$
and the proof is complete.
\end{proof}

\section{Proof of Theorem \ref{boundml}}
It will be convenient to work with the following modification of $m_{t}$. For $r,\tau\in (0,1)$ define
$$m_{\tau, r}f(x):=\sup_{rQ\ni x}(f\chi_Q)^*(\tau|Q|),$$
where the supremum is taken over all cubes $Q$ such that $x\in rQ$.

Suppose that $x\in Q$. Then $x\in r(Q/r)$. Hence,
$$(f\chi_Q)^*(t|Q|)\le (f\chi_{Q/r})^*(tr^n|Q/r|)\le m_{tr^n,r}f(x),$$
and so, for every measurable $f$ and $r\in (0,1)$,
\begin{equation}\label{pes}
m_{t}f(x)\le m_{tr^n,r}f(x).
\end{equation}

\begin{proof}[Proof of Theorem \ref{boundml}] The sufficiency part of this theorem is easy. Indeed, suppose that $m_t$ is bounded on $L^{p(\cdot)}$ for some $t\in (0,1)$.
Take $\la\in (t,1)$. Let ${\mathcal F}$ be a family of pairwise disjoint cubes. Suppose that $E_Q\subset Q, |E_Q|\ge \la|Q|$ for $Q\in {\mathcal F}$. Since $(\chi_{E_Q})^*(t|Q|)=1$, we have
$$\sum_{Q\in {\mathcal F}}t_Q\chi_Q\le m_t\big(\sum_{Q\in F}t_Q\chi_{E_Q}\big).$$
From this and the boundedness of $m_t$ we obtain the ${\mathcal A}_{\infty}$ condition.

Turn now to the necessity part. Suppose that $p(\cdot)\in {\mathcal A}_{\infty}$.
Let $\la\in (0,1)$ be a constant from the ${\mathcal A}_{\infty}$ condition. Next, let $\eta\in (0,1)$ be a constant from Theorem \ref{impl}.
Set $\nu:=\max(\la,\eta)$. Take $t\in (\nu, 1)$, and let us show that $m_t$ is bounded on $L^{p(\cdot)}$.

Let $\|f\|_{L^{p(\cdot)}}=1$. Our goal is to show that there exists $C>0$ independent of $f$ such that
\begin{equation}\label{goal}
\int_{{\mathbb R}^n}\big(m_{t}f(x)\big)^{p(x)}dx\le C.
\end{equation}
Without loss of generality, one can assume that $f$ is bounded and with compact support (otherwise we apply the argument below to $\min(|f|,R)\chi_{\{|x|\le R\}}$
and then let $R\to \infty$).

Take $r\in (0,1)$ such that $\nu<tr^n-(1-r^n)$. By (\ref{pes}), it suffices to prove (\ref{goal}) with $m_{tr^n,r}f$ instead of $m_tf$.

Observe that
$$\|m_{tr^n,r}f\|_{L^{\infty}}\le \|f\|_{L^{\infty}}$$
and, by (\ref{embeq}),
$$\text{supp}(\,m_{tr^n,r}f)\subset \{M\chi_{\{\text{supp}\,(f)\}}>tr^n\}.$$
Therefore, $m_{tr^n,r}f$ is bounded and compactly supported.

For $\ga\in (0,1)$ which will be chosen later and for $k\in {\mathbb Z}$ denote
$$\O_k:=\{x\in {\mathbb R}^n: m_{tr^n,r}f(x)>\ga^k\}.$$
Then
\begin{eqnarray}
&&\int_{{\mathbb R}^n}m_{tr^n,r}f(x)^{p(x)}dx= \sum_{k\in {\mathbb Z}}\int_{\O_{k+1}\setminus \O_k}m_{tr^n,r}f(x)^{p(x)}dx\label{upes}\\
&&\le \sum_{k\in {\mathbb Z}}\int_{\O_{k+1}}\ga^{kp(x)}dx\le (1/\ga)^{p_+}\sum_{k\in {\mathbb Z}}\int_{\O_k}\ga^{kp(x)}dx.\nonumber
\end{eqnarray}

Thus, to complete the proof, it is enough to show that
\begin{equation}\label{ents}
\sum_{k\in {\mathbb Z}}\int_{\O_k}\ga^{kp(x)}dx\le C
\end{equation}
for some universal $C>0$. Let us show first that this series converges. Indeed, the argument in (\ref{upes}) can be reversed,
and, since $m_{tr^n,r}f$ is bounded and compactly supported,
\begin{equation}\label{conv}
\sum_{k\in {\mathbb Z}}\int_{\O_k}\ga^{kp(x)}dx\le C\int_{{\mathbb R}^n}m_{tr^n,r}f(x)^{p(x)}dx<\infty.
\end{equation}

To show (\ref{conv}), observe that for every $k\in {\mathbb Z}$,
\begin{eqnarray*}
\int_{\O_k}\ga^{kp(x)}dx&\le& \int_{\O_k\setminus \O_{k-1}}\ga^{kp(x)}dx+\ga^{p_-}\int_{\O_{k-1}}\ga^{(k-1)p(x)}dx\\
&\le& \int_{\O_k\setminus \O_{k-1}}m_{tr^n,r}f(x)^{p(x)}dx+\ga^{p_-}\int_{\O_{k-1}}\ga^{(k-1)p(x)}dx.
\end{eqnarray*}
From this, for every $M\in {\mathbb N}$,
\begin{eqnarray*}
(1-\ga^{p_-})\sum_{k=-M}^{M-1}\int_{\O_k}\ga^{kp(x)}dx&\le& \sum_{k=-M}^M\int_{\O_k\setminus\O_{k-1}}m_{tr^n,r}f(x)^{p(x)}dx\\
&+&\ga^{p_-}\int_{\O_{-M-1}}\ga^{(-M-1)p(x)}dx.
\end{eqnarray*}
Since $m_{tr^n,r}f$ is bounded, $\O_k=\emptyset$ for all $k\le -M_0$, for some $M_0\in {\mathbb N}$.
Therefore, letting $M\to \infty$ in the above estimate yields (\ref{conv}).

Turn now to proving (\ref{ents}). For every $x\in \O_k$, there exists a cube $Q_x$ such that $x\in rQ_x$ and $(f\chi_{Q_x})^*(tr^n|Q_x|)>\ga_k$. Applying Theorem~\ref{cov}
to the family of cubes $\{Q_x\}_{x\in \Omega_k}$ we obtain that there exists $N=N(n,r)$ families of pairwise disjoint cubes $\{Q_{j,k}^i\}, i=1,\dots,N$ such that
$\O_k\subset\cup_{i=1}^N\cup_jQ_{j,k}^i$ and $(f\chi_{Q_{j,k}^i})^*(tr^n|Q_{j,k}^i|)>\ga^k$.

Denote
$$A_{k,j}^i:=\{x\in Q_{j,k}^i:|f(x)|\ge (f\chi_{Q_{j,k}^i})^*(tr^n|Q_{j,k}^i|)\}.$$
Then $|A_{k,j}^i|\ge tr^n|Q_{j,k}^i|\ge \la|Q_{j,k}^i|$.
Therefore, by the ${\mathcal A}_{\infty}$ condition,
\begin{eqnarray*}
&&\ga^k\|\chi_{Q_{j,k}^i}\|_{L^{p(\cdot)}}\le C\ga^k\|\chi_{A_{k,j}^i}\|_{L^{p(\cdot)}}\\
&&\le C\|(f\chi_{Q_{j,k}^i})^*(tr^n|Q_{j,k}^i|)\chi_{A_{k,j}^i}\|_{L^{p(\cdot)}}\le C\|f\chi_{A_{k,j}^i}\|_{L^{p(\cdot)}}\le C.
\end{eqnarray*}
Hence, $\ga^k/C\in (0,1/\|\chi_{Q_{j,k}^i}\|_{L^{p(\cdot)}}]$.

Next,
\begin{eqnarray*}
\int_{\O_k}\ga^{kp(x)}dx&\le& \sum_{i=1}^N\sum_j\int_{Q_{j,k}^i}\ga^{kp(x)}dx\\
&\le& C^{p_+}\sum_{i=1}^N\sum_j\int_{Q_{j,k}^i}(\ga^k/C)^{p(x)}dx.
\end{eqnarray*}

Denote
$$E_{k,j}^i:=A_{k,j}^i\cap rQ_{j,k}^i.$$
Then $|E_{k,j}^i|\ge \big(tr^n-(1-r^n)\big)|Q_{j,k}^i|\ge \eta|Q_{j,k}^i|$.
Next, if $x\in E_{k,j}^i$, then $x\in rQ_{j,k}^i$ and $(f\chi_{Q_{j,k}^i})^*(tr^n|Q_{j,k}^i|)>\ga^k$, which implies $m_{tr^n,r}f(x)>\ga^k$.
From this, $E_{k,j}^i\subset \O_k$.

Applying Theorem \ref{impl} yields
$$
\sum_j\int_{Q_{j,k}^i}(\ga^k/C)^{p(x)}dx\le 2\Big(\sum_{j}\int_{E_{j,k}^i}(\ga^k/C)^{p(x)}dx+(\ga^k/C)^{\d}\chi_{(0,1)}(\ga^k/C)\Big).
$$
Further,
\begin{eqnarray*}
&&\sum_{j}\int_{E_{j,k}^i}(\ga^k/C)^{p(x)}dx\\
&&\le (1/C)^{p_-}\Big(\int_{\cup_jE_{j,k}^i\setminus \O_{k-1}}\ga^{kp(x)}dx+\int_{\O_{k-1}}\ga^{kp(x)}dx\Big)\\
&&\le (1/C)^{p_-}\Big(\int_{\O_k\setminus \O_{k-1}}|f(x)|^{p(x)}dx+\int_{\O_{k-1}}\ga^{kp(x)}dx\Big).
\end{eqnarray*}

Combining this with the previous estimates yields
\begin{eqnarray*}
\int_{\O_k}\ga^{kp(x)}dx&\le& 2NC^{p_+-p_-}\Big(\int_{\O_k\setminus \O_{k-1}}|f(x)|^{p(x)}dx+\int_{\O_{k-1}}\ga^{kp(x)}dx\Big)\\
&+& 2NC^{p_+-\d}\ga^{k\d}\chi_{(0,1)}(\ga^k/C).
\end{eqnarray*}

Finally we take $\ga\in (0,1)$ such that $2NC^{p_+-p_-}\ga^{p_-}=\frac{1}{2}$. Then
\begin{eqnarray*}
\int_{\O_k}\ga^{kp(x)}dx&\le& 2NC^{p_+-p_-}\int_{\O_k\setminus \O_{k-1}}|f(x)|^{p(x)}dx+\frac{1}{2}\int_{\O_{k-1}}\ga^{(k-1)p(x)}dx\\
&+& 2NC^{p_+-\d}\ga^{k\d}\chi_{(0,1)}(\ga^k/C).
\end{eqnarray*}
From this, since $\int_{{\mathbb R}^n}|f(x)|^{p(x)}dx=1$,
\begin{eqnarray*}
\sum_{k\in {\mathbb Z}}\int_{\O_k}\ga^{kp(x)}dx&\le& 2NC^{p_+-p_-}+\frac{1}{2}\sum_{k\in {\mathbb Z}}\int_{\O_k}\ga^{kp(x)}dx\\
&+&2NC^{p_+-\d}\sum_{k=-k_0}^{\infty}\ga^{k\d},
\end{eqnarray*}
where $k_0\in {\mathbb N}$ depends only on $\ga$ and $C$. Since $\sum_{k\in {\mathbb Z}}\int_{\O_k}\ga^{kp(x)}dx<\infty$, this implies
$$\sum_{k\in {\mathbb Z}}\int_{\O_k}\ga^{kp(x)}dx\le 2\Big(2NC^{p_+-p_-}+2NC^{p_+-\d}\sum_{k=-k_0}^{\infty}\ga^{k\d}\Big),$$
which proves (\ref{ents}), and therefore, the proof is complete.
\end{proof}
\enlargethispage{2cm}

\end{document}